\documentclass[a4paper]{article} 

\usepackage{amsfonts}
\usepackage{amsmath} 
\usepackage{amssymb}
\usepackage{graphicx}
\usepackage{amsthm}
\usepackage{dsfont}    % for making the fat 1 like in characteristic function
\usepackage{microtype} % Apperaently increases the asthetics of the document

\usepackage{caption}
\usepackage{subcaption}
\usepackage{color}
\usepackage{url}

\usepackage[utf8]{inputenc}
\usepackage[english]{babel}

\newtheorem{theorem}{Theorem}[section]
\newtheorem{corollary}[theorem]{Corollary}
\newtheorem{lemma}[theorem]{Lemma}

\theoremstyle{definition}

\theoremstyle{remark}
\newtheorem*{remark}{Remark}
\newtheorem*{acknowledgement}{Acknowledgement}

\newcommand{\onehalf}{\frac{1}{2}}
\newcommand{\N}{\mathbb{N}}
\newcommand{\R}{\mathbb{R}}

\newcommand{\ra}{\rightarrow}

\newcommand{\dx}{\ \mathrm{d}x}

  \date{}
\begin{document}
	\title{On the $L^2$-norm of Gegenbauer polynomials}
	\author{Damir Ferizovi\'{c}\\
		{damir.ferizovic@tugraz.at}  
		\thanks{The author thankfully acknowledges support
	by the Austrian Science Fund (FWF): F5503 ``Quasi-Monte Carlo Methods'' and FWF:  W1230 ``Doctoral School Discrete Mathematics'', and the Austrian Marshall Plan Foundation. }}
\maketitle 
\begin{abstract}
	Gegenbauer, also known as ultra-spherical polynomials   appear often in numerical analysis or interpolation. In the present text we find a recursive formula for, and compute the asymptotic behavior of their $L^2$-norm.\\
	%MSC2020: 33C45, 33F99, 41A60
%	\keywords{Gegenbauer polynomials \and $L^2$-norm \and Asymptotics}
\end{abstract}

\section{Notation and Results}

  Gegenbauer polynomials $\mathcal{C}_{n}^{(\lambda)}$, where  $\lambda\in I_G:=(-\onehalf,0)\cup(0,\infty)$ is called  the index and $n\in\N_0$ is the degree, are the coefficients of following power series expansion in $\alpha$:
  $$ (1-2x\alpha+\alpha^2)^{-\lambda}=\sum_{n=0}^{\infty}\mathcal{C}_{n}^{(\lambda)}(x)\alpha^n.
  $$
  The case $\lambda=0$ is not considered here.
  $\{\mathcal{C}_{n}^{(\lambda)}\}_{n\in\N_0}$ are orthogonal  with respect to the measure $(1-x^2)^{\lambda-1/2}\dx$ over  $[-1,1]$, and by \cite[Eq. 8.930]{Gradshteyn}:
  \begin{equation}\label{eq_FirstGegenbauers}
   \forall\lambda\in I_G:\hspace{1cm}\mathcal{C}_{0}^{(\lambda)}(x)=1,\hspace{1cm} \mathcal{C}_{1}^{(\lambda)}(x)=2\lambda x.
  \end{equation}
  For continuous $f:[0,1]\ra\R$, the  following notation will be used:
  $$
  \|f\|_2^2:=\int_0^1 [f(x)]^2 \dx.
  $$
 We derive an asymptotic formula for $\|\mathcal{C}_{n}^{(\lambda)}\|^2_2$ when $\lambda>0$  in Corollary \ref{cor_maintheorem}. 
 Indeed, one of the key ingredients in \cite{Ferizovic} was the asymptotic nature of $\|\mathcal{C}_{n}^{(2)}\|^2_2$ in $n$, and  the following lemma was proved in \cite[Lemmas 6.1 and 6.2]{Ferizovic}:
\begin{lemma}\label{lem_Ferizovic} Let  $\psi$ denote the  digamma function 	and $\gamma$  the Euler-Mascheroni constant. Then the Gegenbauer polynomials satisfy for  $n\geq 2$:
	\begin{equation*}
	\begin{split}
		\big\|\sqrt{1-x^2}\ \mathcal{C}_{n-2}^{(2)}\big\|_2^2
	&=\tfrac{1}{16}(2n^2-1)\big(\psi(n+\tfrac{1}{2})+\gamma+\log(4)\big) -\tfrac{1}{8}n^2,\\
	\big\|\mathcal{C}_{n-2}^{(2)}\big\|^2_2
	&=\tfrac{1}{16}n^4+\tfrac{1}{64}(4n^2-1)\big(\psi(n+\tfrac{1}{2})+\gamma+\log(4)\big) -\tfrac{5}{32}n^2.
	\end{split}
	\end{equation*}
\end{lemma}
The following result of Corollary 5.2 from \cite{Dette} will prove to be indispensable.
%We always can compute in principle $\|\sqrt{1-x^2}\ \mathcal{C}_{n}^{(\lambda+1)}\|_2^2$  by means of $\|\mathcal{C}_{k}^{(\lambda)}\|^2_2$:
\begin{theorem}[Dette \cite{Dette}]The Gegenbauer polynomials satisfy for $\lambda\in I_G$
	\begin{equation}\label{eq_Dette}
	\left(\frac{n}{2 \lambda}\right)^2\big[\mathcal{C}_{n}^{(\lambda)}(x)\big]^2+
	(1-x^2)\big[\mathcal{C}_{n-1}^{(\lambda+1)}(x)\big]^2=
	\sum_{k=0}^{n-1}\frac{\lambda+k}{\lambda}\big[\mathcal{C}_{k}^{(\lambda)}(x)\big]^2.
	\end{equation}
\end{theorem}
Our main theorem is as follows and we will use it to derive the asymptotic behavior of $\|\mathcal{C}_{n}^{(\lambda)}\|^2_2$.
\begin{theorem}[Main Result]\label{thm_L2norm}
	The  Gegenbauer polynomials satisfy for $\lambda\in I_G$ and $n>1$:
	\begin{equation*}
	\big\|\mathcal{C}_{n-2}^{(\lambda+1)}\big\|_2^2=\frac{n^2-2\lambda n}{2^4\lambda^3}\big[\mathcal{C}_{n}^{(\lambda)}(1)\big]^2+\frac{ n(2n+1)}{2^3\lambda^2}\big\|\mathcal{C}_{n}^{(\lambda)}\big\|_2^2-\sum_{k=0}^{n-1}\frac{\lambda+k}{2^2\lambda^2}\big\|\mathcal{C}_{k}^{(\lambda)}\big\|_2^2.
	\end{equation*}
\end{theorem}

\begin{corollary}\label{cor_maintheorem}
	Let $\mathcal{B}(x,y)$ denote the beta function. The following asymptotic formulas in $n$ hold for $\lambda\in(0,1)$ and $\delta(\lambda):=\max\{4\lambda-1,2\lambda\}$:
	\begin{align*}
	\big\|\mathcal{C}_{n}^{(\lambda)}\big\|_2^2&<\	\mathcal{B}\big(1-\lambda,\tfrac{1}{2}\big)\frac{2^{1-2\lambda}}{\Gamma(\lambda)^2}\frac{1}{n^{2-2\lambda}},\\
	\big\|\mathcal{C}_{n}^{(\lambda+1)}\big\|_2^2&=\frac{n^{4\lambda}}{4\lambda\Gamma(2\lambda+1)^2}+O\big(n^{\delta(\lambda)}\big),\\
	\big\|\sqrt{1-x^2}\ \mathcal{C}_{n-1}^{(\lambda+1)}\big\|_2^2&<\frac{\	\mathcal{B}\big(1-\lambda,\tfrac{1}{2}\big)}{\Gamma(\lambda+1)^2}\frac{n^{2\lambda}}{2^{1+2\lambda}}+O\big(n^{\delta(\lambda)-1}\big).
	\end{align*}
	The following asymptotic formulas hold for $\lambda>1$:
	\begin{align*}
	\big\|\mathcal{C}_{n-2}^{(\lambda+1)}\big\|_2^2&=\frac{n^{4\lambda}}{4\lambda\Gamma(2\lambda+1)^2}+
	\frac{\lambda-1}{\Gamma(2\lambda+1)^2}n^{4\lambda-1}
	+O(n^{4\lambda-2}),\\
	\big\|\sqrt{1-x^2}\ \mathcal{C}_{n-1}^{(\lambda+1)}\big\|_2^2&=\frac{2\lambda-1}{4(\lambda-1)\Gamma(2\lambda+1)^2}n^{4\lambda-2}+O(n^{\delta(\lambda-1)+2}).
	\end{align*}
\end{corollary}
The identity $2\cdot\|\mathcal{C}_{n}^{(1)}\|_2^2=\psi(n+\tfrac{3}{2})+\gamma+\log(4)$ is given by  \cite[Eq. 14]{Ferizovic}. 
  \section{Ingredients for the Proof of the Theorem}
  In this section we collect known results concerning Gegenbauer polynomials for later reference and the reader's 
  convenience, and  we derive some technical lemmas in Subsection \ref{subsec_ident}  to prove Theorem \ref{thm_L2norm}. To avoid repetition, we will assume $\lambda\in I_G$ for the rest of the text if not stated otherwise.
  Note first that
  \begin{equation}\label{eq_chebyshevT_basics}
  \begin{array}{rll}
  \dfrac{\mathrm{d}}{\mathrm{d}x}\mathcal{C}_{n+1}^{(\lambda)}(x)&=2\lambda\ \mathcal{C}_{n}^{(\lambda+1)}(x)& \mbox{ \cite[Eq. 8.935]{Gradshteyn}},\\
  \mathcal{C}_{n}^{(\lambda)}(1)&=\frac{\Gamma(n+2\lambda)}{\Gamma(2\lambda)n!}=%\binom{2\lambda+n-1}{n}
  \frac{\prod_{j=1}^{n}(2\lambda+n-j)}{n!}& \mbox{ \cite[Eq. 8.937]{Gradshteyn}};\\
  \end{array}
  \end{equation}
and $\mathcal{C}_{n}^{(\lambda)}(1)$ is the maximum on [-1,1] for $\lambda>0$ by \cite[Eq. 7.33.1]{Szego}. Also, by \eqref{eq_chebyshevT_basics}:
\begin{align}
		(n+2)\ \mathcal{C}_{n+2}^{(\lambda)}(x)&=2\lambda\Big( x\ \mathcal{C}_{n+1}^{(\lambda+1)}(x)-\mathcal{C}_{n}^{(\lambda+1)}(x)\Big)&\mbox{\cite[Eq. 8.933.2]{Gradshteyn}},\label{eq_GradshteynGegenbauerRecursion}\\
		(n+\lambda)\ \mathcal{C}_{n}^{(\lambda)}(x)&=\lambda\Big(\mathcal{C}_{n}^{(\lambda+1)}(x)-\mathcal{C}_{n-2}^{(\lambda+1)}(x)\Big)&\mbox{\cite[Eq. 8.939.6]{Gradshteyn}}.\label{eq_GegenbauerDiff}
\end{align}
%	\begin{equation}\label{eq_GradshteynGegenbauerRecursion}
%	(n+2)\mathcal{C}_{n+2}^{(\lambda)}(x)=2\lambda x\mathcal{C}_{n+1}^{(\lambda+1)}(x)-2\lambda\mathcal{C}_{n}^{(\lambda+1)}(x)\hspace{1cm}\mbox{\cite[Eq. 8.933.2]{Gradshteyn}};
%	\end{equation}
%	\begin{equation}\label{eq_GegenbauerDiff}
%	\lambda\mathcal{C}_{n}^{(\lambda+1)}(x)-\lambda\mathcal{C}_{n-2}^{(\lambda+1)}(x)=(n+\lambda)\mathcal{C}_{n}^{(\lambda)}(x)\hspace{1.2cm}\mbox{\cite[Eq. 8.939.6]{Gradshteyn}}.
%\end{equation}	
\subsection{Identities for Gegenbauer polynomials}\label{subsec_ident}
\begin{lemma}\label{lem_GegenbauerSumDiff}
	The Gegenbauer polynomials satisfy following identities:
	\begin{equation*}
		\begin{split}
		\mathcal{C}_{n}^{(\lambda+1)}(x)+\mathcal{C}_{n-2}^{(\lambda+1)}(x)&=2x\ \mathcal{C}_{n-1}^{(\lambda+1)}(x)+\mathcal{C}_{n}^{(\lambda)}(x),\hspace{2,4cm}(\star)\\
		\int_0^1\big[\mathcal{C}_{n}^{(\lambda+1)}(x)\big]^2-\big[\mathcal{C}_{n-2}^{(\lambda+1)}(x)\big]^2\dx&=\frac{n+\lambda}{2\lambda^2}\Big(\big[\mathcal{C}_{n}^{(\lambda)}(1)\big]^2+	(2\lambda-1)\big\|\mathcal{C}_{n}^{(\lambda)}\big\|_2^2\Big).\\
		\end{split}
	\end{equation*}
\end{lemma}
\begin{proof}
  First we use \eqref{eq_GegenbauerDiff}; then apply \eqref{eq_GradshteynGegenbauerRecursion} to the right-hand side below proving  $(\star)$:
	\begin{equation*}
	\begin{split}
	\mathcal{C}_{n}^{(\ell)}(x)+\mathcal{C}_{n-2}^{(\ell)}(x)&=\frac{n+\lambda}{\lambda}\mathcal{C}_{n}^{(\lambda)}(x)+2x\ \mathcal{C}_{n-1}^{(\ell)}(x)-\Big(2x\ \mathcal{C}_{n-1}^{(\ell)}(x)- 2\mathcal{C}_{n-2}^{(\ell)}(x)\Big),
	\end{split}
	\end{equation*}
	where $\ell:=\lambda+1$.
		Next we obtain by the binomial theorem with \eqref{eq_GegenbauerDiff}, $(\star)$ and \eqref{eq_chebyshevT_basics}
		\begin{equation*}
		\begin{split}
		\big[\mathcal{C}_{n}^{(\lambda+1)}(x)\big]^2-\big[\mathcal{C}_{n-2}^{(\lambda+1)}(x)\big]^2&=\frac{n+\lambda}{\lambda}\mathcal{C}_{n}^{(\lambda)}(x)\Big(2x\mathcal{C}_{n-1}^{(\lambda+1)}(x)+\mathcal{C}_{n}^{(\lambda)}(x)\Big)\\
		&=\frac{n+\lambda}{\lambda}\Big(\frac{x}{2\lambda} \frac{\mathrm{d}}{\mathrm{d}x}\big[\mathcal{C}_{n}^{(\lambda)}(x)\big]^2+\big[\mathcal{C}_{n}^{(\lambda)}(x)\big]^2\Big).
		\end{split}
	\end{equation*}
 Integration by parts then finishes the argument. 
\end{proof}

\begin{lemma}\label{lem_GegenbauerSumSquares}
	The Gegenbauer polynomials satisfy the following identity:
		\begin{equation*}
	\begin{split}
	\int_0^1x^2\big[\mathcal{C}_{n+1}^{(\lambda+1)}(x)\big]^2+\big[\mathcal{C}_{n}^{(\lambda+1)}(x)\big]^2&\dx+\frac{1}{2\lambda}\int_0^1(1-x^2)\big[\mathcal{C}_{n+1}^{(\lambda+1)}(x)\big]^2\dx\\
	&\hspace{-2cm}=\frac{(n+2)^2}{8 \lambda^3}\left[\mathcal{C}_{n+2}^{(\lambda)}(1)\right]^2+\frac{2\lambda-1}{2\lambda}\frac{(n+2)^2}{4 \lambda^2}\big\|\mathcal{C}_{n+2}^{(\lambda)}\big\|_2^2.
	\end{split}
	\end{equation*}
\end{lemma}
\begin{proof}
	Let $n=2m$. By Lemma \ref{lem_GegenbauerSumDiff} and a telescoping sum argument:
		\begin{equation*}
		\begin{split}
		\big\|\mathcal{C}_{n}^{(\lambda+1)}\big\|^2_2-\big\|\mathcal{C}_{0}^{(\lambda+1)}\big\|^2_2&=\sum_{j=1}^m\frac{2j+\lambda}{2\lambda^2}\Big(\big[\mathcal{C}_{2j}^{(\lambda)}(1)\big]^2+	(2\lambda-1)\big\|\mathcal{C}_{2j}^{(\lambda)}\big\|_2^2\Big),\\
		\big\|\mathcal{C}_{n+1}^{(\lambda+1)}\big\|^2_2-\big\|\mathcal{C}_{1}^{(\lambda+1)}\big\|^2_2&=\sum_{j=1}^m\frac{2j+1+\lambda}{2\lambda^2}\Big(\big[\mathcal{C}_{2j+1}^{(\lambda)}(1)\big]^2	+	(2\lambda-1)\big\|\mathcal{C}_{2j+1}^{(\lambda)}\big\|_2^2\Big).
		\end{split}
\end{equation*}
	Using \eqref{eq_FirstGegenbauers} and summing up, and an application of Dette's result \eqref{eq_Dette} yields:
		\begin{equation*}
		\begin{split}
		\int_0^1&\big[\mathcal{C}_{n+1}^{(\lambda+1)}(x)\big]^2+\big[\mathcal{C}_{n}^{(\lambda+1)}(x)\big]^2\dx\\
		&=\frac{4}{3}(\lambda+1)^2+1+\frac{1}{2\lambda}\sum_{j=2}^{n+1}\frac{j+\lambda}{\lambda}\big[\mathcal{C}_{j}^{(\lambda)}(1)\big]^2+\frac{2\lambda-1}{2\lambda}\sum_{j=2}^{n+1}\frac{j+\lambda}{\lambda}\big\|\mathcal{C}_{j}^{(\lambda)}\big\|_2^2\\
		&=\frac{(n+2)^2}{8 \lambda^3}\big[\mathcal{C}_{n+2}^{(\lambda)}(1)\big]^2+\frac{2\lambda-1}{2\lambda}\sum_{j=0}^{n+1}\frac{j+\lambda}{\lambda}\big\|\mathcal{C}_{j}^{(\lambda)}\big\|_2^2\\
		&=\frac{(n+2)^2}{8 \lambda^3}\big[\mathcal{C}_{n+2}^{(\lambda)}(1)\big]^2+\frac{2\lambda-1}{2\lambda}\Big(\frac{(n+2)^2}{4 \lambda^2}\big\|\mathcal{C}_{n+2}^{(\lambda)}\big\|_2^2+\big\|\sqrt{1-x^2}\ \mathcal{C}_{n+1}^{(\lambda+1)}\big\|_2^2\Big).\\
		\end{split}
	\end{equation*}
	The case $n+1=2m$ is analogous.
\end{proof}

\begin{lemma}\label{lem_GegenbauerXSquared}
	The Gegenbauer polynomials satisfy the following identity:
		\begin{equation*}
	\int_0^1x^2\big[\mathcal{C}_{n+1}^{(\lambda+1)}(x)\big]^2-\big[\mathcal{C}_{n}^{(\lambda+1)}(x)\big]^2\dx=
	\frac{n+2}{4 \lambda^2}\Big(\big[\mathcal{C}_{n+2}^{(\lambda)}(1)\big]^2
	-(n+3)\big\|\mathcal{C}_{n+2}^{(\lambda)}\big\|_2^2\Big).
	\end{equation*}
%	\begin{equation*}
%	\big\|x\mathcal{C}_{n+1}^{(\lambda+1)}\big\|_2^2=\frac{n+2}{8\lambda^2}\Big(\big[\mathcal{C}_{n+2}^{(\lambda)}(1)\big]^2-\big\|\mathcal{C}_{n+2}^{(\lambda)}\big\|_2^2\Big)+\int_0^1x\mathcal{C}_{n+1}^{(\lambda+1)}(x)\mathcal{C}_{n}^{(\lambda+1)}(x)\dx.
%	\end{equation*}
	%	\begin{equation*}
	%	\begin{split}
	%	\int_0^1x^2\big[\mathcal{C}_{n+1}^{(\lambda+1)}(x)\big]^2&\dx=\frac{n+2}{8\lambda^2}\Big(\big[\mathcal{C}_{n+2}^{(\lambda)}(1)\big]^2-\int_0^1\big[\mathcal{C}_{n+2}^{(\lambda)}(x)\big]^2\dx\Big)\\
	%	&\hspace{2,4cm}+\int_0^1x\mathcal{C}_{n}^{(\lambda+1)}(x)\mathcal{C}_{n+1}^{(\lambda+1)}(x)\dx.\\
	%	\end{split}
	%	\end{equation*}
\end{lemma}
\begin{proof}
	Note first that
	by \eqref{eq_GradshteynGegenbauerRecursion} and by quadratic completion
	\begin{equation}\label{eq_ProductGegenb}
	\begin{split}
		2\frac{n+2}{2\lambda}\mathcal{C}_{n+2}^{(\lambda)}(x)\mathcal{C}_{n}^{(\lambda+1)}(x)&=2x\mathcal{C}_{n+1}^{(\lambda+1)}(x)\mathcal{C}_{n}^{(\lambda+1)}(x)-2\big[\mathcal{C}_{n}^{(\lambda+1)}(x)\big]^2\\
	&\hspace{-2.5cm}=x^2\big[\mathcal{C}_{n+1}^{(\lambda+1)}(x)\big]^2-\big[\mathcal{C}_{n}^{(\lambda+1)}(x)\big]^2-\Big(x\mathcal{C}_{n+1}^{(\lambda+1)}(x)-\mathcal{C}_{n}^{(\lambda+1)}(x) \Big)^2.
	\end{split}
	\end{equation}
	Hence by the binomial theorem and again by \eqref{eq_GradshteynGegenbauerRecursion}
	\begin{equation*}
	\begin{split}
	&2\int_0^1x^2\big[\mathcal{C}_{n+1}^{(\lambda+1)}(x)\big]^2-\big[\mathcal{C}_{n}^{(\lambda+1)}(x)\big]^2\dx\\
	&=\frac{n+2}{\lambda}\int_0^1\Big(x\mathcal{C}_{n+1}^{(\lambda+1)}(x)+\mathcal{C}_{n}^{(\lambda+1)}(x)\Big)\mathcal{C}_{n+2}^{(\lambda)}(x)\dx\\
	&=\frac{n+2}{\lambda}\int_0^1\frac{x}{4\lambda}\frac{\mathrm{d}}{\mathrm{d}x}\big[\mathcal{C}_{n+2}^{(\lambda)}(x)\big]^2+\mathcal{C}_{n}^{(\lambda+1)}(x)\mathcal{C}_{n+2}^{(\lambda)}(x)\dx\\
	&=\frac{n+2}{4\lambda^2}\Big(\big[\mathcal{C}_{n+2}^{(\lambda)}(1)\big]^2-\int_0^1\big[\mathcal{C}_{n+2}^{(\lambda)}(x)\big]^2\dx\Big)+2\frac{n+2}{2\lambda}\int_0^1\mathcal{C}_{n}^{(\lambda+1)}(x)\mathcal{C}_{n+2}^{(\lambda)}(x)\dx\\
	%&=\frac{n+2}{8\lambda^2}\Big(\big[\mathcal{C}_{n+2}^{(\lambda)}(1)\big]^2-\big\|\mathcal{C}_{n+2}^{(\lambda)}\big\|_2^2\Big)+\int_0^1x\mathcal{C}_{n}^{(\lambda+1)}(x)\mathcal{C}_{n+1}^{(\lambda+1)}(x)-\big[\mathcal{C}_{n}^{(\lambda+1)}(x)\big]^2\dx\\
	\end{split}
	\end{equation*}
	which proves the result when we substitute \eqref{eq_ProductGegenb} and use \eqref{eq_GradshteynGegenbauerRecursion} one last time.
\end{proof}

%\begin{lemma}\label{lem_ProductGegenbauer}
%	The Gegenbauer polynomials satisfy following identity for $\lambda\neq0$:
%	\begin{equation*}
%	\begin{split}
%	&\int_0^12x\mathcal{C}_{n+1}^{(\lambda+1)}(x)\mathcal{C}_{n}^{(\lambda+1)}(x)\dx\\
%	&=\frac{(n+2)^2}{8 \lambda^3}\big[\mathcal{C}_{n+2}^{(\lambda)}(1)\big]^2-\frac{1}{2\lambda}\left(\frac{(n+2)^2}{4 \lambda^2}\big\|\mathcal{C}_{n+2}^{(\lambda)}\big\|_2^2+\big\|\sqrt{1-x^2}\mathcal{C}_{n+1}^{(\lambda+1)}\big\|_2^2\right).
%	\end{split}
%	\end{equation*}
%\end{lemma}
%\begin{proof}
%	Clearly we have 
%	\begin{equation*}
%		2x\mathcal{C}_{n+1}^{(\lambda+1)}\mathcal{C}_{n}^{(\lambda+1)}=x^2\big[\mathcal{C}_{n+1}^{(\lambda+1)}\big]^2+\big[\mathcal{C}_{n}^{(\lambda+1)}\big]^2-\Big(x\mathcal{C}_{n+1}^{(\lambda+1)}-\mathcal{C}_{n}^{(\lambda+1)} \Big)^2,\hspace{1cm} (\star\star)
%		\end{equation*}
%		thus by  Lemma \ref{lem_GegenbauerSumSquares} and by \eqref{eq_GradshteynGegenbauerRecursion} we obtain 
%%		\begin{equation*}
%%		\begin{split}
%%		\int_0^1(\star\star)\dx&=\frac{(n+2)^2}{8 \lambda^3}\big[\mathcal{C}_{n+2}^{(\lambda)}(1)\big]^2+\frac{2\lambda-1}{2\lambda}\frac{(n+2)^2}{4 \lambda^2}\big\|\mathcal{C}_{n+2}^{(\lambda)}\big\|_2^2\\
%%		&\hspace{1cm}-\frac{1}{2\lambda}\int_0^1(1-x^2)\big[\mathcal{C}_{n+1}^{(\lambda+1)}(x)\big]^2\dx-\frac{(n+2)^2}{4\lambda^2}\big\|\mathcal{C}_{n+2}^{(\lambda)}\big\|_2^2.
%%		\qedhere
%%		\end{split}
%%	\end{equation*}
%\end{proof}

\section{Proof of the Main Results}
	\begin{proof}[Proof of Theorem \ref{thm_L2norm}] Subtract the left hand sides of Lemma \ref{lem_GegenbauerSumSquares} and Lemma~\ref{lem_GegenbauerXSquared}:
			\begin{equation*}
		\begin{split}
		2\big\|\mathcal{C}_{n}^{(\lambda+1)}\big\|^2_2&+\frac{1}{2\lambda}	\big\|\sqrt{1-x^2}\ \mathcal{C}_{n+1}^{(\lambda+1)}\big\|_2^2=\Big(\frac{(n+2)^2}{8 \lambda^3}-
		\frac{n+2}{4 \lambda^2}\Big)\big[\mathcal{C}_{n+2}^{(\lambda)}(1)\big]^2\\
		&\hspace{0.5cm}+\Big(\frac{(n+2)^2}{4 \lambda^2}+
		\frac{(n+2)(n+3)}{4 \lambda^2}\Big)\big\|\mathcal{C}_{n+2}^{(\lambda)}\big\|_2^2-\frac{1}{2\lambda}\frac{(n+2)^2}{4 \lambda^2}\big\|\mathcal{C}_{n+2}^{(\lambda)}\big\|_2^2;
		\end{split}
		\end{equation*}
		 an application of Dette's formula \eqref{eq_Dette} then gives the desired expression.
	\end{proof}
\begin{corollary}\label{cor_OneMinusXsquaredNorm}
	The Gegenbauer polynomials satisfy the following identity:
	\begin{align*}
		\big\|\sqrt{1-x^2}\ \mathcal{C}_{n-1}^{(\lambda+1)}\big\|_2^2&\\
		&\hspace{-2.5cm}=
\big[\mathcal{C}_{n}^{(\lambda)}(1)\big]^2\ \frac{n+2\lambda}{n+1}\frac{1-2\lambda}{2^3\lambda^2}+\frac{ (n+1)(2n+3)}{2^3\lambda^2}\big\|\mathcal{C}_{n+1}^{(\lambda)}\big\|_2^2-\big\|\mathcal{C}_{n}^{(\lambda)}\big\|_2^2\ \frac{ n+2\lambda}{2^3\lambda^2}.
	\end{align*}
\end{corollary}
\begin{proof}
	We use Lemma \ref{lem_GegenbauerXSquared}, add zero and obtain with Theorem \ref{thm_L2norm}
	\begin{align*}
			\frac{n}{4 \lambda^2}&\Big(\big[\mathcal{C}_{n}^{(\lambda)}(1)\big]^2
		-(n+1)\big\|\mathcal{C}_{n}^{(\lambda)}\big\|_2^2\Big)+	\int_0^1(1-x^2)\big[\mathcal{C}_{n-1}^{(\lambda+1)}(x)\big]^2\dx\\
		&=\int_0^1\big[\mathcal{C}_{n-1}^{(\lambda+1)}(x)\big]^2-\big[\mathcal{C}_{n-2}^{(\lambda+1)}(x)\big]^2\dx\\
		&=\frac{(n+1)^2-2\lambda (n+1)}{2^4\lambda^3}\big[\mathcal{C}_{n+1}^{(\lambda)}(1)\big]^2+\frac{ (n+1)(2n+3)}{2^3\lambda^2}\big\|\mathcal{C}_{n+1}^{(\lambda)}\big\|_2^2\\
		&\hspace{1cm}-\sum_{k=0}^{n}\frac{\lambda+k}{2^2\lambda^2}\big\|\mathcal{C}_{k}^{(\lambda)}\big\|_2^2\\
		&\hspace{1cm}-\frac{n^2-2\lambda n}{2^4\lambda^3}\big[\mathcal{C}_{n}^{(\lambda)}(1)\big]^2-\frac{ n(2n+1)}{2^3\lambda^2}\big\|\mathcal{C}_{n}^{(\lambda)}\big\|_2^2+\sum_{k=0}^{n-1}\frac{\lambda+k}{2^2\lambda^2}\big\|\mathcal{C}_{k}^{(\lambda)}\big\|_2^2\\
		&=\big[\mathcal{C}_{n}^{(\lambda)}(1)\big]^2\Big(	\frac{(n+1)^2-2\lambda (n+1)}{2^4\lambda^3}\frac{(n+2\lambda)^2}{(n+1)^2}	-\frac{n^2-2\lambda n}{2^4\lambda^3}\Big)\\
		&\hspace{1cm}+\frac{ (n+1)(2n+3)}{2^3\lambda^2}\big\|\mathcal{C}_{n+1}^{(\lambda)}\big\|_2^2-\big\|\mathcal{C}_{n}^{(\lambda)}\big\|_2^2\Big(\frac{ n(2n+1)}{2^3\lambda^2}+\frac{\lambda+n}{2^2\lambda^2}\Big)\\	&=\big[\mathcal{C}_{n}^{(\lambda)}(1)\big]^2\ \frac{2n^2+3n+2\lambda-2\lambda n-4\lambda^2}{2^3\lambda^2(n+1)}\\
		&\hspace{1cm}+\frac{ (n+1)(2n+3)}{2^3\lambda^2}\big\|\mathcal{C}_{n+1}^{(\lambda)}\big\|_2^2-\big\|\mathcal{C}_{n}^{(\lambda)}\big\|_2^2\Big(\frac{ 2n^2+3n+2\lambda}{2^3\lambda^2}\Big).
	\end{align*}
We re-order to obtain the result.
%	\begin{align*}
%	\big\|\sqrt{1-x^2}\ \mathcal{C}_{n-1}^{(\lambda+1)}\big\|_2^2&=
%		\big[\mathcal{C}_{n}^{(\lambda)}(1)\big]^2\ \frac{n+2\lambda}{n+1}\frac{1-2\lambda}{2^3\lambda^2}-\big\|\mathcal{C}_{n}^{(\lambda)}\big\|_2^2\ \frac{ n+2\lambda}{2^3\lambda^2}\\
%		&\hspace{3cm}+\frac{ (n+1)(2n+3)}{2^3\lambda^2}\big\|\mathcal{C}_{n+1}^{(\lambda)}\big\|_2^2. \qedhere
%	\end{align*}
\end{proof}
\begin{remark}
  For our asymptotic analysis we will need the  following identity, which follows from the proof of Theorem \ref{thm_L2norm} and Corollary \ref{cor_OneMinusXsquaredNorm}:
  \begin{equation}\label{rem_AsymptoticHelp}
  	\begin{split}
  	\big\|\mathcal{C}_{n-2}^{(\lambda+1)}\big\|_2^2&=\frac{n^2-2\lambda n}{2^4\lambda^3}\big[\mathcal{C}_{n}^{(\lambda)}(1)\big]^2+\frac{ 2n^2(4\lambda-1)+n(4\lambda+1)+2\lambda}{2^5\lambda^3}\big\|\mathcal{C}_{n}^{(\lambda)}\big\|_2^2\\
  	&\hspace{1.2cm}-\big[\mathcal{C}_{n}^{(\lambda)}(1)\big]^2\ \frac{n+2\lambda}{n+1}\frac{1-2\lambda}{2^5\lambda^3}-\frac{ (n+1)(2n+3)}{2^5\lambda^3}\big\|\mathcal{C}_{n+1}^{(\lambda)}\big\|_2^2.
  	\end{split}
  \end{equation}
  		Using following asymptotic form, see \cite{Tricomi}: For $|z|\ra\infty$ and $\alpha,\beta\geq0$:
  	\begin{equation}\label{eq_Tricomi}
  	\frac{\Gamma(z+\alpha)}{\Gamma(z+\beta)}=z^{\alpha-\beta}\Big(1+\frac{(\alpha-\beta)(\alpha+\beta-1)}{2z}+O(|z|^{-2})\Big),
  	\end{equation}
  	we obtain by \eqref{eq_chebyshevT_basics} for $\lambda>0$:
  	\begin{equation}\label{eq_asymptoticGegenbauer1}
  	\Gamma(2\lambda)^2\cdot 	\big[\mathcal{C}_{n}^{(\lambda)}(1)\big]^2=n^{4\lambda-2}+2\lambda(2\lambda-1)n^{4\lambda-3}		+O(n^{4\lambda-4}).
  	\end{equation}
\end{remark}
\begin{proof}[Proof of Corollary \ref{cor_maintheorem}]  
	We will write $\|\mathcal{C}_{n}^{(\lambda)}\|_2^2=\Theta(n^{\Phi(\lambda)})$ if there are  some constants $c_1,c_2>0$ such that $c_1n^{\Phi(\lambda)}\leq\|\mathcal{C}_{n}^{(\lambda)}\|_2^2 \leq c_2n^{\Phi(\lambda)}$ for all $n$ big enough. First we use \eqref{rem_AsymptoticHelp} to show by induction that  $\Phi(\lambda)$ exists for $\lambda>1$, and that $\big[\mathcal{C}_{n}^{(\lambda)}(1)\big]^2=\Theta(n^{\Phi(\lambda)+2})$.	
\paragraph{The case $\lambda=m\in\N_{>1}$:}  Lemma \ref{lem_Ferizovic} gives the result for  $\lambda=2$,
% and with \eqref{rem_AsymptoticHelp}, it follows that $\big[\mathcal{C}_{n}^{(3)}(1)\big]^2=\Theta(n^{\Phi(3)+2})$.
and if it holds for $m$, then with \eqref{rem_AsymptoticHelp} and abuse of notation we have:
\begin{equation*}
\big\|\mathcal{C}_{n-2}^{(m+1)}\big\|_2^2=n^2\  \Theta\big(n^{\Phi(m)+2}\big)+n^2\ \Theta\big(n^{\Phi(m)}\big)+\Theta\big(n^{\Phi(m)+2}\big)+n^2\ \Theta\big(n^{\Phi(m)}\big).
\end{equation*}
This proves the claim as it shows that  $\|\mathcal{C}_{n}^{(m+1)}\|_2^2=\Theta(n^{\Phi(m)+4})$, but by \eqref{eq_chebyshevT_basics}:
\begin{equation}\label{eq_GegenIndex}
\mathcal{C}_{n}^{(\lambda+1)}(1)=\frac{(2\lambda+n+1)(2\lambda+n)}{2\lambda(2\lambda+1)}
\mathcal{C}_{n}^{(\lambda)}(1)=\Theta\big(n^2\mathcal{C}_{n}^{(\lambda)}(1)\big),
\end{equation}
which, when squared and $\lambda=m$, is of order $\Phi(m)+6.$ 

\paragraph{The case $\lambda\in (m,m+1)$ for $m\in\N$:} For  $\lambda\in (0,1)$ and $\theta\in[0,\pi]$:
$$\sin(\theta)^{\lambda}\big|\mathcal{C}_{n}^{(\lambda)}(\cos(\theta))\big|< \frac{2^{1-\lambda}}{\Gamma(\lambda)} n^{\lambda-1}\hspace{1cm}\mbox{see \cite[Eq. 7.33.5]{Szego}}.$$
We square this inequality, multiply by $\sin(\theta)^{1-2\lambda}$ and integrate:
$$\big\|\mathcal{C}_{n}^{(\lambda)}\big\|_2^2<\frac{2^{2-2\lambda}}{\Gamma(\lambda)^2} n^{2\lambda-2}\int_0^{\pi/2} \sin(\theta)^{1-2\lambda}\ \mathrm{d}\theta=
\mathcal{B}\big(1-\lambda,\tfrac{1}{2}\big)\frac{2^{1-2\lambda}}{\Gamma(\lambda)^2}n^{2\lambda-2}; $$
where we used a change of variables $\theta=\arcsin(x)$ and $\mathcal{B}(x,y)$ is the beta function.
This in combination with \eqref{eq_asymptoticGegenbauer1} and \eqref{rem_AsymptoticHelp} gives for $\delta=\max\{4\lambda-1,2\lambda\}$:
\begin{equation}\label{eq_resttermlambdasmall}
\big\|\mathcal{C}_{n}^{(\lambda+1)}\big\|_2^2=\frac{n^{4\lambda}}{2^4\lambda^3\Gamma(2\lambda)^2}+O\big(n^{\delta}\big).
\end{equation}
Thus for $\lambda\in (0,1)$: $\Phi(\lambda+1)=4\lambda$, and $[\mathcal{C}_{n}^{(\lambda+1)}(1)]^2=\Theta\big(n^{4\lambda+2}\big)$ by \eqref{eq_asymptoticGegenbauer1}, which finishes the case for the interval  $ (1,2)$ and we use induction with \eqref{rem_AsymptoticHelp} and \eqref{eq_GegenIndex}.

Thus  the two leading terms in the asymptotic form of $\|\mathcal{C}_{n}^{(\lambda+1)}\|_2^2$ are in the expansion of  $ \mathcal{C}_{n}^{(\lambda)}(1)$ when $\lambda>1$; using once more  \eqref{eq_asymptoticGegenbauer1} and \eqref{rem_AsymptoticHelp} yields
\begin{equation*}
\frac{n^2-2\lambda n}{2^4\lambda^3}\big[\mathcal{C}_{n}^{(\lambda)}(1)\big]^2=
\frac{n^{4\lambda}}{4\lambda\Gamma(2\lambda+1)^2}+
\frac{2\lambda(2\lambda-2)}{4\lambda\Gamma(2\lambda+1)^2}n^{4\lambda-1}
+O(n^{4\lambda-2}).
\end{equation*}

The asymptotic of the rest term of $\|\mathcal{C}_{n}^{(\lambda+1)}\|_2^2$ follows by \eqref{rem_AsymptoticHelp}, equation  \eqref{eq_resttermlambdasmall} and induction for non-integer $\lambda\in\R_{>2}$ or else  Lemma \ref{lem_Ferizovic} and induction  when $\lambda\in\N_{> 2}$. Now, these asymptotic formulas in combination with Corollary \ref{cor_OneMinusXsquaredNorm} and \eqref{eq_asymptoticGegenbauer1} will finish the argument: As an illustration, let $1<\lambda\leq 2$, the cases $0<\lambda<1$ and $\lambda>2$ are similar; let $\rho=\max\{4\lambda-3,2\lambda\}$, then 
\begin{align*}
\big\|\sqrt{1-x^2}\ \mathcal{C}_{n-1}^{(\lambda+1)}\big\|_2^2&=
\frac{1-2\lambda}{2^3\lambda^2}\big[\mathcal{C}_{n}^{(\lambda)}(1)\big]^2+\frac{ 2n^2}{2^3\lambda^2}\|\mathcal{C}_{n+1}^{(\lambda)}\|_2^2+O(n^{\rho})\\
&=
\frac{1-2\lambda}{2^3\lambda^2}\frac{n^{4\lambda-2}}{\Gamma(2\lambda)^2}+\frac{ 2n^2}{2^3\lambda^2}	\frac{n^{4\lambda-4}}{4(\lambda-1)\Gamma(2\lambda-1)^2}+O(n^{\rho})\\
&=
\frac{n^{4\lambda-2}}{2^3\lambda^2\Gamma(2\lambda-1)^2}\Big(\frac{1-2\lambda}{(2\lambda-1)^2}+\frac{1}{2(\lambda-1)}\Big)+O(n^{\rho})\\
&=
\frac{n^{4\lambda-2}}{2^3\lambda^2\Gamma(2\lambda-1)^2}\frac{1}{(2\lambda-1)(\lambda-1)2}+O(n^{\rho}).\qedhere
\end{align*}
\end{proof}
\begin{remark}\label{rem_lamdaEqualOne}
One can use Lemma \ref{lem_Ferizovic}, Corollary \ref{cor_OneMinusXsquaredNorm} and identity \eqref{rem_AsymptoticHelp} to find  formulas  for $\|\sqrt{1-x^2}\ \mathcal{C}_{n}^{(m)}\|_2^2$ and $\|\mathcal{C}_{n}^{(m)}\|_2^2$ where $m\in\N_{>1}$.
\end{remark}

\begin{acknowledgement}
	The help of J. Brauchart,  P. Grabner and J. Thuswaldner is gratefully appreciated; who proof read the manuscript, made useful remarks on   presentation and suggested to generalize Corollary \ref{cor_maintheorem} from $\lambda\in\onehalf\N$ to $\lambda>0$.
\end{acknowledgement}

\end{document}